\documentclass[a4paper,12pt]{amsart}
\usepackage{hyperref}
\usepackage{amssymb}
\usepackage{color}

\allowdisplaybreaks

\numberwithin{equation}{section}

\newtheorem{thm}{Theorem}[section]
\newtheorem{prop}[thm]{Proposition}
\newtheorem{lem}[thm]{Lemma}
\newtheorem{cor}[thm]{Corollary}

\theoremstyle{definition}

\theoremstyle{remark}
\newtheorem{rmk}[thm]{Remark}
\newtheorem{ex}[thm]{Example}

\newcommand{\CC}{\mathbb{C}}
\newcommand{\NN}{\mathbb{N}}
\newcommand{\RR}{\mathbb{R}}
\newcommand{\TT}{\mathbb{T}}

\def\Oo{\mathcal{O}}
\def\T{\mathcal{T}}

\newcommand{\Aut}{\operatorname{Aut}}

\newcommand{\im}{{\operatorname{Im}}}
\title[KMS states on the algebras of higher-rank graphs]{\boldmath{KMS states on $C^*$-algebras associated to\\ higher-rank graphs}}
\author[an Huef]{Astrid an Huef}
\author[Laca]{Marcelo Laca}
\author[Raeburn]{Iain Raeburn}
\author[Sims]{Aidan Sims}

\address{Astrid an Huef and Iain Raeburn\\ Department of Mathematics and Statistics\\University of Otago\\PO Box 56\\Dunedin 9054\\New Zealand}
\email{astrid@maths.otago.ac.nz, iraeburn@maths.otago.ac.nz}

\address{Marcelo Laca\\ Department of Mathematics and Statistics\\
University of Victoria\\
Victoria, BC V8W 3P4\\
Canada}
\email{laca@math.uvic.ca}

\address{Aidan  Sims\\ School of Mathematics and Applied Statistics\\
University of Wollongong\\NSW 2522\\Australia}
\email{asims@uow.edu.au}

\date{\today}

\subjclass[2010]{46L30, 46L55}
\thanks{This research has been supported by the University of Otago, the Marsden Fund of the Royal Society of New Zealand, the Natural Sciences and Engineering Research Council of Canada, and the Australian Research Council.}

\begin{document}
\maketitle

\begin{abstract}

Consider a higher-rank graph of rank $k$. Both the Cuntz-Krieger algebra and the Toeplitz-Cuntz-Krieger algebra of the graph carry natural gauge actions of the torus $\TT^k$, and restricting these gauge actions to one-parameter subgroups of $\TT^k$ gives dynamical systems involving actions of the real line. We study the KMS states of these dynamical systems. We find that for large inverse temperatures $\beta$, the simplex of KMS$_\beta$ states on the Toeplitz-Cuntz-Krieger algebra has dimension $d$ one less than the number of vertices in the graph. We also show that there is a preferred dynamics for which there is a critical inverse temperature $\beta_c$: for $\beta$ larger than $\beta_c$, there is a $d$-dimensional simplex of KMS states; when $\beta=\beta_c$ and the one-parameter subgroup is dense, there is a unique KMS state, and this state factors through the Cuntz-Krieger algebra. As in previous studies for $k=1$, our main tool is the Perron-Frobenius theory for irreducible nonnegative matrices, though here we need a version of the theory for commuting families of matrices.

\end{abstract}

\section{Introduction}

Every Cuntz-Krieger algebra carries a natural gauge action of the circle, which we can lift to $\RR$ to get a natural dynamics. Enomoto, Fujii and Watatani proved in 1984 that when $\Oo_A$ is a simple Cuntz-Krieger algebra, this dynamics admits a single KMS state, which occurs at inverse temperature $\ln\rho(A)$, where $\rho(A)$ is the spectral radius of $A$ \cite{EnomotoFujiiEtAl:MJ84}. Their result has been extended to the Cuntz-Krieger algebras of infinite matrices \cite{EL}, and more recently to the graph algebras of finite graphs with sources \cite{KW, aHLRSk=1}. We now know also that the Toeplitz-Cuntz-Krieger algebras of matrices and graphs have a much richer supply of KMS states \cite{EL, LN, aHLRSk=1}.

Higher-rank graphs (or \emph{$k$-graphs}) were introduced by Kumjian and Pask \cite{KP} as combinatorial models for the higher-rank
Cuntz-Krieger algebras of \cite{RobertsonSteger:JRAM99}, and their $C^*$-algebras have generated a great deal of interest. In particular, Davidson and
Yang have made a detailed study of $C^*$-algebras associated to
finite $2$-graphs with a single vertex \cite{DY, Yang:xx09, Yang}, and recently Yang \cite{Yang:xx09} has studied KMS
states for a 1-parameter subgroup of the gauge action of the $2$-torus on the $C^*$-algebra $C^*(\Lambda)$ of such a graph $\Lambda$. She has shown that, if $\Lambda$ has $m$ blue edges and $n$ red ones, and if $\ln m$ is not a rational multiple of $\ln n$, then $C^*(\Lambda)$ has a unique KMS state.

Here we consider the dynamics obtained by restricting the gauge action of $\TT^k$ to one-parameter subgroups of $\TT^k$, and we compute the KMS states on both the Toeplitz algebra and $C^*$-algebra of a finite $k$-graph $\Lambda$ with no sources. For the Toeplitz algebra $\T C^*(\Lambda)$ and large inverse temperatures $\beta$, the KMS$_\beta$ states form a simplex of dimension $|\Lambda^0|-1$ (Theorem~\ref{mainthmk}). For smaller $\beta$, it matters which $1$-parameter subgroup we choose, and there is a preferred dynamics for which we obtain the best possible results. To describe it, we consider the vertex matrices $A_i$ of the subgraphs containing just the edges of degree $e_i$ (that is, the individually coloured subgraphs of the skeleton); for the graphs in the previous paragraph, for example, we recover $A_1=(m)$ and $A_2=(n)$. Then the preferred dynamics is the restriction of the gauge action to the subgroup
$\{\prod_{i=1}^k\rho(A_i)^{it}:t\in \RR\}$.

For the preferred dynamics, Theorem~\ref{mainthmk} describes a $(|\Lambda^0|-1)$-dimensional simplex of KMS$_\beta$ states for every $\beta>1$. When the numbers $\ln\rho(A_i)$ are rationally independent, there is a unique KMS$_1$ state, which factors through a state of the graph algebra $C^*(\Lambda)$ (Theorem~\ref{criticalbeta}).

Our approach is a refinement of the one developed in \cite{LR} and \cite{LRR}, which we previously applied to the $C^*$-algebras of finite graphs with sources in \cite{aHLRSk=1}. The first task in this approach is to find a characterisation of KMS$_\beta$ states which will enable us to recognise them in the wild. When the $1$-parameter subgroup determining the action is dense in $\TT^k$, this is relatively straightforward (Proposition~\ref{idKMSbetak}): the outcome is, roughly, that a KMS state $\phi$ is determined by its values on the vertex projections $\{q_v:v\in \Lambda^0\}$  of the generating Toeplitz-Cuntz-Krieger family $\{t,q\}$. The Toeplitz-Cuntz-Krieger relations then impose restrictions on the vector $m^\phi:=(\phi(q_v))\in [0,\infty)^k$, which amount to saying that $m^\phi$ is subinvariant for all the matrices $A_i$, in the sense of Perron-Frobenius theory \cite{Seneta}. When $\phi$ is a state on $C^*(\Lambda)$, this subinvariance becomes invariance (Proposition~\ref{propsofm}). For graphs such that the matrices $A_i$ are irreducible, the Perron-Frobenius theory puts lower bounds on the possible inverse temperatures $\beta$ (Corollary~\ref{restrictbeta}), and these lower bounds are all the same for the preferred dynamics (see Remark~\ref{pref}).

In \S\ref{sec:largebeta}, we consider the KMS$_\beta$ states on the Toeplitz algebra for the action determined by an arbitrary subgroup of $\TT^k$. We show that, even when this subgroup is not dense, KMS$_\beta$ states are still invariant under the full gauge action of $\TT^k$ (Theorem~\ref{improvedidKMS}; see also Remark~\ref{extrasym}). This invariance was easy to establish in the case $k=1$ (see Proposition~2.1(a) in \cite{aHLRSk=1}), but seems to require a rather delicate analysis for $k>1$. The extra invariance plays a crucial role when we prove surjectivity of our parametrisation of the KMS states on the Toeplitz algebra (Theorem~\ref{mainthmk}). As in \cite{aHLRSk=1}, this parametrisation includes concrete descriptions of the KMS states in terms of the path representation of $\T C^*(\Lambda)$ on $\ell^2(\Lambda)$.

In \S\ref{sec:CK-KMS}, we consider the preferred dynamics $\alpha$ under the assumption that the numbers $\ln\rho(A_i)$ are rationally independent. We show that there is then a unique KMS$_1$ state on $(\T C^*(\Lambda),\alpha)$, that this state factors through a state of the graph algebra $C^*(\Lambda)$, and that the resulting state of $C^*(\Lambda)$ is the only KMS state of $(C^*(\Lambda),\alpha)$. We show by example that uniqueness can fail if the $\ln\rho(A_i)$ are rationally dependent. We finish with a short analysis of the ground states of our systems.

\section{Background}

\subsection{Perron-Frobenius for commuting matrices}\label{PFback} Let $X$ be a finite set. A matrix $A \in M_X(\CC)$ is
\emph{irreducible} if for all $x,y \in X$ there exists $n \in \NN$ such that $A^n(x,y) \not= 0$. We
say that a matrix is positive (non-negative) if all its entries are positive (non-negative). A vector in
$[0,\infty)^X$ is \emph{subinvariant} for $A \in M_X(\CC)$ and $\lambda \in \RR$ if $Am\leq \lambda m$ pointwise.

Let $A$ be an irreducible matrix with non-negative entries. The Perron-Frobenius theorem \cite[Theorem~1.5]{Seneta} says that $A$ has an eigenvalue $\rho_A>0$ with a positive
eigenvector $x$, that $\rho_A$ is equal to the spectral radius $\rho(A)$ of $A$, and  that $\rho_A$
is a simple root of the characteristic polynomial of $A$. The last statement in \cite[Theorem~1.6]{Seneta} implies that every non-negative eigenvector of $A$ has eigenvalue $\rho(A)$. Thus $A$ has a unique
non-negative eigenvector $x$ whose $\ell^1$-norm is $1$; we call $x$ the \emph{unimodular Perron-Frobenius eigenvector} of $A$. 

We need the following
version of the Perron-Frobenius theorem for commuting matrices.

\begin{lem}\label{lem:common F-B evctr}
Suppose that $A$ and $B$ are commuting irreducible $n \times n$ matrices with non-negative entries.
Then their unimodular Perron-Frobenius eigenvectors are equal.
\end{lem}
\begin{proof}
Let $x$ be the unimodular Perron-Frobenius eigenvector for $A$. We have
\[
ABx = BAx = \rho(A)Bx.
\]
Hence $Bx$ is an eigenvector of $A$. Since $B$ has non-negative entries and $x$ is a positive
vector, $Bx$ is  non-negative. The Perron-Frobenius theorem for $A$ implies that $Bx$ is a
scalar multiple of $x$, so $x$ is an eigenvector of $B$ also. Since $x$ is positive, applying \cite[Theorem~1.6]{Seneta} to $B$ shows that $Bx=\rho(B)x$, and since $x$ has
$\ell^1$-norm 1, $x$ is also the unique
unimodular non-negative eigenvector for $B$.
\end{proof}

\subsection{Higher-rank graphs and their $C^*$-algebras}
Let $\Lambda$ be a $k$-graph with vertex set $\Lambda^0$ and degree functor $d:\Lambda\to \NN^k$, as in \cite{KP}. We are only interested in $k$-graphs which are finite in the sense that $\Lambda^n:=d^{-1}(n)$ is finite for all $n\in \NN^k$. We use the convention that $v\Lambda^n w$, for example, denotes the set $\{\lambda\in \Lambda^n:r(\lambda)=v\text{ and }s(\lambda)=w\}$. For a pair $(\lambda,\mu)$ of paths in
$\Lambda$, we write $\Lambda^{\min}(\lambda,\mu)$ for the set of pairs $(\eta,\zeta)$ in $\Lambda$
such that $\lambda\eta=\mu\zeta$ and $d(\lambda\eta)=d(\lambda)\vee d(\mu)$.

For $1\leq i\leq k$, we write $A_i$ for the \emph{vertex matrix} with entries
$A_i(v,w)=|v\Lambda^{e_i}w|$. Since $(A_iA_j)(v,w)$ is the number of paths of degree $e_i+e_j$ from $w$ to $v$, the
factorisation property implies that $A_iA_j=A_jA_i$. So it makes sense to define
$A^n:=\prod_{i=1}^k A_i^{n_i}$ for $n\in \NN^k$, and then 
$A^n(v,w)$ is the number $|v\Lambda^nw|$ of paths from $w$ to $v$ of degree $n$. 

\begin{lem}\label{resolventk}
Suppose that $\Lambda$ is a finite $k$-graph with no sources, and that $\beta\in [0,\infty)$ and
$r\in (0,\infty)^k$ satisfy $\beta r_i > \ln \rho(A_i)$ for $1\leq i\leq k$. View $\NN^k$ as a directed
set with $m\leq n\Longleftrightarrow m_i\leq n_i$ for all $i$. Then the series $\sum_{n\in\NN^k}
e^{-\beta r\cdot n}A^n$ converges in the operator norm to $\prod_{i=1}^k(1-e^{-\beta
r_i}A_i)^{-1}$.
\end{lem}
\begin{proof}
Since the vertex matrices $A_i$ commute, the $N$th partial sum is
\[
\sum_{0\leq n\leq N} e^{-\beta r\cdot n}A^n
    = \sum_{0\leq n\leq N}\prod_{i=1}^k e^{-\beta r_in_i}A_i^{n_i}
    = \prod_{i=1}^k\Big(\sum_{n_i=0}^{N_i}e^{-\beta r_i}A_i^{n_i}\Big).
\]
For each $i$ we have $\beta r_i>\rho(A_i)$, and hence $\sum_{n_i=0}^{N_i}e^{-\beta
r_in_i}A_i^{n_i}$ converges to $(1-e^{-\beta r_i}A_i)^{-1}$ in the operator norm as $N_i\to \infty$
(by, for example, \cite[Lemma~VII.3.4]{DS}). Thus as $N\to \infty$ in $\NN^k$, which means
$N_i\to \infty$ for all $i$, the product converges in the operator norm to
$\prod_{i=1}^k(1-e^{-\beta r_i}A_i)^{-1}$.
\end{proof}

Following the spirit of \cite[\S7]{RS}, a \emph{Toeplitz-Cuntz-Krieger $\Lambda$-family} $\{T,Q\}$ consists
of partial isometries $\{T_\lambda:\lambda\in\Lambda\}$ such that
\begin{itemize}
\item[(T1)] $\{Q_v:=T_v:v\in \Lambda^0\}$ are mutually orthogonal projections;
\item[(T2)] $T_\lambda T_\mu=T_{\lambda\mu}$ whenever $s(\lambda)=r(\mu)$;
\item[(T3)] $T_\lambda^*T_\lambda=Q_{s(\lambda)}$ for all $\lambda$;
\item[(T4)] for all $v\in \Lambda^0$ and $n\in\NN^k$, we have
\begin{equation*}
Q_v\geq\sum_{\lambda\in v\Lambda^n} T_\lambda T_\lambda^*;
\end{equation*}
\item[(T5)] for all $\mu,\nu\in \Lambda$, we have (interpreting any empty sums as $0$)
\begin{equation}\label{Nicacov}
T_\mu^*T_\nu=\sum_{(\eta,\zeta)\in \Lambda^{\min}(\mu,\nu)} T_\eta T_\zeta^*.
\end{equation}
\end{itemize}
The relation (T4) implies that the range projections $\{T_\lambda T_\lambda^*:d(\lambda)=n\}$ are mutually orthogonal for each $n$ (by the argument of \cite[Corollary~1.2]{aHLRSk=1}). A Toeplitz-Cuntz-Krieger $\Lambda$-family is a \emph{Cuntz-Krieger $\Lambda$-family} if in addition  we have
\begin{itemize}\item[(CK)] $Q_v=\sum_{\lambda\in v\Lambda^n} T_\lambda T_\lambda^*$ for all $v\in \Lambda^0$ and $n\in \NN^k$.
\end{itemize}

\begin{rmk}
The relations (T1)--(T3) and (CK) imply (T5) \cite[Lemma~3.1]{KP}, but in the absence of (CK) we
need to impose (T5) to ensure that the $C^*$-algebra generated by $\{T_\lambda\}$ is spanned by
elements of the form $T_\lambda T_\mu^*$. Proposition~6.4 of \cite{RS} says that (T5) is equivalent
to Nica covariance of an associated representation of a product system of Hilbert bimodules over
$C(\Lambda^0)$, as in \cite{F}.
\end{rmk}

The Toeplitz algebra $\T C^*(\Lambda)$ is generated by a universal Toeplitz-Cuntz-Krieger family
$\{t_\lambda,q_v\}$. It can be constructed as the Nica-Toeplitz algebra of a product system of Hilbert
bimodules, as in \cite[Corollary~7.5]{RS}. The universal property shows that there is a gauge
action $\gamma$ of $\TT^k$ on $\T C^*(\Lambda)$ such that $\gamma_z(t_\lambda)=z^{d(\lambda)}t_\lambda$ (using
multi-index notation).

For a specific example of a Toeplitz-Cuntz-Krieger family, consider the usual orthonormal basis
$\{h_\lambda:\lambda\in\Lambda\}$ for $\ell^2(\Lambda)$, and the operators $Q_v$, $T_\mu$ such that
\[
T_\mu h_\lambda=\begin{cases} h_{\mu\lambda}&\text{if $s(\mu)=r(\lambda)$}\\
0&\text{otherwise} \end{cases}
\quad\text{and}\quad Q_v h_\lambda=\begin{cases} h_{\lambda}&\text{if $v=r(\lambda)$}\\
0&\text{otherwise.} \end{cases}
\]
It is proved in \cite[Example~7.4]{RS} that $(T,Q)$ is a Toeplitz-Cuntz-Krieger $\Lambda$-family,
and in \cite[Corollary~7.7]{RS} that the associated representation $\pi_{T,Q}$ of $\T C^*(\Lambda)$
is faithful. We call $\pi_{T,Q}$ the \emph{path representation} of $\T C^*(\Lambda)$.

\section{Properties of KMS states}\label{sec:charKMS}

Suppose that $\alpha$ is an action of $\RR$ on a $C^*$-algebra $A$. An element $a\in A$ is \emph{analytic} for $\alpha$ if $t\mapsto\alpha_t(a)$ is the restriction to $\RR$ of an analytic function $z\mapsto\alpha_z(a)$ defined on all of $\CC$.  A state $\phi$ on $A$ is a \emph{KMS state with inverse temperature $\beta\in (0,\infty)$} (or \emph{KMS$_\beta$ state}) of $(A,\alpha)$ if 
\begin{equation}\label{defKMS}
\phi(ab)=\phi(b\alpha_{i\beta}(a))\quad\text{for all analytic elements $a, b$;}
\end{equation}
$\phi$ is a \emph{ground state} of $(A,\alpha)$ if $z\mapsto \phi(a\alpha_z(b))$ is bounded in the upper half-plane $\im z>0$ for all analytic $a, b$. By convention, the KMS$_0$ states are the invariant traces, and the KMS$_\infty$ states are the weak* limits of sequences of KMS$_{\beta_n}$ states as $\beta_n\to \infty$. KMS states and ground states are always $\alpha$-invariant \cite[Proposition~8.12.4]{P}.

We know on general grounds that the analytic elements always form a dense subalgebra of $A$ \cite[page~363]{P}. But in practice it seems to be easy to find plenty of analytic elements, and this is all we have to do: by \cite[Proposition~8.12.3]{P}, it suffices to check the KMS and ground-state conditions on a set of analytic elements which span a dense subspace of $A$. In the systems $(\T C^*(\Lambda),\alpha)$ of interest to us, we have $\alpha_t(t_\mu t_\nu^*)=e^{itr\cdot(d(\mu)-d(\nu))}t_\mu t_\nu^*$ for some $r\in (0,\infty)^k$, and $z\mapsto e^{izr\cdot(d(\mu)-d(\nu))}t_\mu t_\nu^*$ is an analytic function which extends $t\mapsto \alpha_t(t_\mu t_\nu^*)$. So it suffices for us to check the KMS condition \eqref{defKMS} on the elements $t_\mu t_\nu^*$. 

We begin our investigation of KMS states on $\T C^*(\Lambda)$ by giving a characterisation of KMS states that will make them easier to identify.

\begin{prop}\label{idKMSbetak}
Suppose that $\Lambda$ is a finite $k$-graph with no sources. Let $r \in (0,\infty)^k$, let $\gamma : \TT^k \to \Aut(\T C^*(\Lambda))$ be the gauge action, and
define $\alpha : \RR \to \Aut(\T C^*(\Lambda))$ by $\alpha_t = \gamma_{e^{itr}}$. Suppose that $\beta\in
[0,\infty)$ and $\phi$ is a state on $\T C^*(\Lambda)$.
\begin{enumerate}
\item\label{prepreak} If $\phi$ is a KMS$_\beta$ state of $(\T C^*(\Lambda),\alpha)$, then
\begin{equation}\label{propKMSbetak}
    \phi(t_\mu t_\nu^*)=\delta_{\mu,\nu}e^{-\beta r\cdot d(\mu)}\phi(q_{s(\mu)})\quad\text{for all $\mu,\nu\in \Lambda$ with $d(\mu)=d(\nu)$.}
    \end{equation}
\item\label{preak} If
    \begin{equation}\label{charKMSbetak}
    \phi(t_\mu t_\nu^*)=\delta_{\mu,\nu}e^{-\beta r\cdot d(\mu)}\phi(q_{s(\mu)})\quad\text{for all $\mu,\nu\in \Lambda$,}
    \end{equation}
then $\phi$ is a KMS$_\beta$ state of $(\T
    C^*(\Lambda),\alpha)$.     If $r \in (0,\infty)^k$ has rationally
independent coordinates, then $\phi$ is a KMS$_\beta$ state of $(\T
    C^*(\Lambda),\alpha)$ if and only if \eqref{charKMSbetak} holds.
\item\label{groundk} A state $\phi$ on $\T C^*(\Lambda)$ is a ground state of $(\T
    C^*(\Lambda), \alpha)$ if and only if
    \begin{equation}\label{chargroundk}
    \phi(t_\mu t_\nu^*) = 0 \text{ whenever $r\cdot d(\mu) > 0$ or $r\cdot d(\nu) > 0$.}
    \end{equation}
\end{enumerate}
\end{prop}

\begin{proof}[Proof of Proposition~\ref{idKMSbetak}\,\eqref{prepreak}]
If $d(\mu)=d(\nu)$, then $t_\nu^*t_\mu=\delta_{\mu,\nu}q_{s(\mu)}$, and the KMS condition gives
\[
\phi(t_\mu t_\nu^*)=e^{-\beta r\cdot d(\mu)}\phi(t_\nu^*t_\mu)=\delta_{\mu,\nu}e^{-\beta r\cdot d(\mu)}\phi(q_{s(\mu)}),
\]
which is \eqref{propKMSbetak}.
\end{proof}

To prove part~(\ref{preak}) we need the following property of the partial order on $\NN^k$.

\begin{lem}\label{minimal}
Suppose that $m,n,p,q\in \NN^k$ satisfy $m+p=n+q$. Then $m+p=m\vee n$ if and only if $p\wedge q=0$.
\end{lem}

\begin{proof}
First suppose $m+p=m\vee n$. Then for each $i$,
\[
p_i=(m\vee n)_i-m_i=\max(m_i,n_i)-m_i
=\begin{cases}0&\text{if $n_i< m_i$}\\n_i-m_i&\text{if $n_i\geq m_i$,}\end{cases}
\]
and
\[
q_i=m_i-n_i+p_i
=\begin{cases}m_i-n_i&\text{if $n_i<m_i$}\\0&\text{if $n_i\geq m_i$.}\end{cases}
\]
So for all $i$, $\min(p_i,q_i)=0$, which means $p\wedge q=0$.

Conversely, suppose that $\min(p_i,q_i)=0$ for all $i$. If $p_i\not=0$, then $q_i=0$,
$m_i+p_i=n_i$, and $m_i+p_i=\max(m_i,n_i)$. If $p_i=0$, then $m_i=m_i+p_i=n_i+q_i$, $m_i\geq n_i$,
and $m_i+p_i=m_i=\max(m_i,n_i)$. So $m+p=m\vee n$.
\end{proof}

\begin{proof}[Proof of Proposition~\ref{idKMSbetak}\,\eqref{preak}]
Suppose that $\phi$ satisfies \eqref{charKMSbetak}, and consider two spanning elements $t_\mu t_\nu^*$ and $t_\sigma t_\tau^*$ with $s(\mu)=s(\nu)$ and $s(\sigma)=s(\tau)$. Then (T5) implies that
\begin{equation*}
\phi(t_\mu t_\nu^*t_\sigma t_\tau^*)=\sum_{(\alpha,\eta)\in\Lambda^{\min}(\nu,\sigma)} \phi(t_{\mu\alpha}t_{\tau\eta}^*).
\end{equation*}
The formula \eqref{charKMSbetak} says that the $(\alpha,\eta)$-summand vanishes unless
$\mu\alpha=\tau\eta$, and hence
\begin{equation}\label{onewayround}
\phi(t_\mu t_\nu^*t_\sigma t_\tau^*)=\sum_{\{(\alpha,\eta)\in\Lambda^{\min}(\nu,\sigma)\,:\,\mu\alpha=\tau\eta\}} e^{-\beta r\cdot d(\mu\alpha)}\phi(q_{s(\alpha)}).
\end{equation}
Similarly
\begin{equation}\label{tother}
\phi(t_\sigma t_\tau^*t_\mu t_\nu^*)=\sum_{\{(\gamma,\zeta)\in\Lambda^{\min}(\tau,\mu)\,:\,\sigma\gamma=\nu\zeta\}} e^{-\beta r\cdot d(\sigma\gamma)}\phi(q_{s(\gamma)}).
\end{equation}

We will next show that the indexing sets in the sums in \eqref{onewayround} and \eqref{tother} are
closely related. To see this, suppose that $(\alpha,\eta)\in\Lambda^{\min}(\nu,\sigma)$ satisfies
$\mu\alpha=\tau\eta$. Since $(\alpha,\eta)\in\Lambda^{\min}(\nu,\sigma)$  we have $d(\alpha)\wedge
d(\eta)=0$, because otherwise the extension is not minimal. Applying Lemma~\ref{minimal} to the
equality $d(\mu)+d(\alpha)=d(\tau)+d(\eta)$ shows that $d(\mu)+d(\alpha)=d(\mu)\vee d(\tau)$; hence $(\eta,\alpha)$ belongs to $\Lambda^{\min}(\tau,\mu)$. The situation is symmetric, so we deduce that the map $(\alpha,\eta)\mapsto
(\eta, \alpha)$ is a bijection of the index set in \eqref{onewayround} onto the index set in
\eqref{tother}. 

Fix $(\alpha, \eta)$ in the index set in \eqref{onewayround}. Since $s(\alpha)=s(\eta)$, the projections $q_{s(\alpha)}$ and $q_{s(\eta)}$ are
the same. Thus to verify the KMS condition, it suffices for us to see that the coefficient
$e^{-\beta r\cdot d(\mu\alpha)}$ of $\phi(q_{s(\alpha)})$ is the same as the coefficient of
$\phi(q_{s(\eta)})$ in the sum for
\[
\phi\big(t_\sigma t_\tau^*\alpha_{i\beta}(t_\mu t_\nu^*)\big)=e^{-\beta r\cdot (d(\mu)-d(\nu))}\phi(t_\sigma t_\tau^*t_\mu t_\nu^*),
\]
which is
\[
e^{-\beta r\cdot (d(\mu)-d(\nu))}e^{-\beta r\cdot d(\sigma\eta)}=e^{-\beta r\cdot (d(\mu)-d(\nu))}e^{-\beta r\cdot d(\nu\alpha)}.
\]
Now the identity
\[
d(\mu)-d(\nu)+d(\nu\alpha)=d(\mu)-d(\nu)+d(\nu)+d(\alpha)=d(\mu\alpha)
\]
gives the required equality of coefficients.  Thus $\phi$ is a KMS${}_\beta$ state.

Next suppose that $r\in (0,\infty)^k$ has rationally independent coordinates.
Let  $\phi$ be a KMS$_\beta$ state.  Take $\mu,\nu\in \Lambda$ with $s(\mu)=s(\nu)$. If $d(\mu)=d(\nu)$, then the KMS condition gives
\[
\phi(t_\mu t_\nu^*)=e^{-\beta r\cdot d(\mu)}\phi(t_\nu^*t_\mu)=\delta_{\mu,\nu}e^{-\beta r\cdot d(\mu)}\phi(q_{s(\mu)}).
\]
So we suppose that $d(\mu)\not= d(\nu)$.  Then two applications of the KMS condition give
\begin{equation}\label{twotimeKMS}
\phi(t_\mu t_\nu^*)=e^{-\beta r\cdot d(\mu)}\phi(t_\nu^*t_\mu)=e^{-\beta r\cdot d(\mu)}e^{\beta r\cdot d(\nu)}\phi(t_\mu t_\nu^*).
\end{equation} 
Since the entries of $r$ are rationally independent and $d(\mu)\not= d(\nu)$, we have $r\cdot d(\mu)\not=r\cdot d(\nu)$; hence $e^{\beta r\cdot d(\mu)}\not= e^{\beta r\cdot d(\nu)}$, and \eqref{twotimeKMS} implies that $\phi(t_\mu t_\nu^*)=0$. Thus $\phi$ satisfies \eqref{charKMSbetak}.
\end{proof}

\begin{proof}[Proof of Proposition~\ref{idKMSbetak}\,\eqref{groundk}]
For every state $\phi$, all paths $\sigma$, $\tau$, $\mu$ and $\nu$, and all $z=x+iy\in \CC$, we have
\begin{align}\label{groundbded}
|\phi(t_\sigma t_\tau^* \alpha_{z}(t_\mu t_\nu^*))|&= |e^{i(x+iy)r\cdot (d(\mu)-d(\nu))}\phi(t_\sigma t_\tau^* t_\mu t_\nu^*)|\\
&= e^{-yr\cdot (d(\mu)-d(\nu))}|\phi(t_\sigma t_\tau^* t_\mu t_\nu^*)|.\notag
\end{align}
Suppose that $\phi$ is a ground state. Then $|\phi(t_\mu \alpha_{x+iy}(t_\nu^*))|=e^{yr\cdot d(\nu)}|\phi(t_\mu t_\nu^*)|$ is bounded on the upper half-plane $y>0$, and hence $\phi(t_\mu t^*_\nu) = 0$ whenever $r\cdot d(\nu) > 0$. Since $\phi(t_\nu t^*_\mu) = \overline{\phi(t_\mu t^*_\nu)}$, symmetry implies that $\phi(t_\mu t^*_\nu) = 0$ whenever $r\cdot d(\mu) > 0$. Hence $\phi$ satisfies~\eqref{chargroundk}. 

Now suppose that
$\phi$ satisfies~\eqref{chargroundk}. In view of \eqref{groundbded}, the Toeplitz-Cuntz-Krieger relation (T5) implies that 
\begin{equation}\label{useT5}
|\phi(t_\sigma t_\tau^* \alpha_{x+iy}(t_\mu t_\nu^*))|= e^{-yr\cdot (d(\mu)-d(\nu))}\Big|\sum_{(\alpha,\eta)\in \Lambda^{\min}(\tau,\mu)}\phi(t_{\sigma \alpha} t_{\nu \eta}^*)\Big|.
\end{equation}
By \eqref{chargroundk}, 
\begin{align*}
\phi(t_{\sigma\alpha}t^*_{\nu\eta})\neq 0\quad&\Longrightarrow\quad r\cdot d(\sigma\alpha)=0=r\cdot d(\nu\eta)\\
&\Longrightarrow\quad d(\sigma\alpha)=0=d(\nu\eta)\quad\text{(since  $r_i>0$)}\\
& \Longrightarrow\quad d(\sigma)+d(\alpha)=0=d(\nu)+d(\eta).
\end{align*}
So it is only possible to have a nonzero summand in \eqref{useT5} if $d(\sigma)=d(\nu)=d(\alpha)=d(\eta)=0$,  and then $\tau=\mu$ because    $(\alpha,\eta)\in \Lambda^{\min}(\tau,\mu)$. In that case, there is only one summand, and 
\[
|\phi(t_\sigma t_\tau^* \alpha_{x+iy}(t_\mu t_\nu^*))|= e^{-yr\cdot d(\mu)}|\phi(t_{\mu}^*t_{\mu})|=e^{-yr\cdot d(\mu)}\phi(q_{s(\mu)})
\]
is bounded on the upper half-plane $y>0$. Thus $\phi$ is a ground state.
\end{proof}

\begin{ex}
The condition $r_i>0$ was crucial in the proof of part~\eqref{groundk}, and if some $r_i<0$ there may not be many ground states. For example, suppose that $\Lambda$ is a $2$-graph with one vertex $v$, and that $|\Lambda^{e_1}|=|\Lambda^{e_2}|=1$ (in general, we call edges in $\Lambda^{e_1}$ blue and those in $\Lambda^{e_2}$ red).
Take $r=(-1,1)$. We claim that $(\T C^*(\Lambda),\alpha)$ has no ground states. To see this, let $\phi$ be a state of $\T C^*(\Lambda)$ and consider the blue edge $e$. Then
\[
\big|\phi(t_e^*\alpha_{x+iy}(t_e))\big|=e^{-yr\cdot d(e)}\phi(t_e^*t_e)=e^{-yr\cdot e_1}\phi(q_{s(e)})=e^{y}\phi(q_v)=e^y,
\]
because $q_v$ is the identity of $\T C^*(\Lambda)$. Thus $z\mapsto \phi(t_e^*\alpha_{z}(t_e))$ is not bounded on the upper half-plane, and $\phi$ is not a ground state.

It is, however, possible to find states of $\T C^*(\Lambda)$ satisfying \eqref{chargroundk}. Indeed, consider  the path representation $\pi_{T,Q}$ of $\T C^*(\Lambda)$ on $\ell^2(\Lambda)$, the basis element $h_e$ associated to $e$, and the vector state $\phi:a\mapsto (\pi_{T,Q}(a)h_e\,|\,h_e)$. If $\nu\in \Lambda$ has $r\cdot d(\nu)>0$, then $d(\nu)_2>0$; thus $\nu$ has a factorisation $f\nu'$ with $f$ red, and $\pi_{T,Q}(t_\mu t_\nu^*)h_e=\pi_{T,Q}(t_\mu t_{\nu'}^*)T_f^*h_e=0$, so $\phi(t_\mu t_\nu^*)=0$. A similar argument works if $r\cdot d(\mu)>0$, so $\phi$ satisfies \eqref{chargroundk}.
\end{ex}

\section{KMS states and the subinvariance relation}

The gist of Proposition~\ref{idKMSbetak}, and especially the formula \eqref{charKMSbetak}, is that a KMS state $\phi$ on $\T C^*(\Lambda)$ is determined by the numbers $\{\phi(q_v):v\in \Lambda^0\}$. We now look at how these numbers are related to the vertex matrices $A_i$ of $\Lambda$.

\begin{prop}\label{propsofm}
Suppose that $\Lambda$ is a finite $k$-graph with no sources. Let $r \in (0,\infty)^k$, let $\gamma : \TT^k \to \Aut(\T C^*(\Lambda))$ be the gauge action, and
define $\alpha : \RR \to \Aut(\T C^*(\Lambda))$ by $\alpha_t = \gamma_{e^{itr}}$. Suppose that $\beta\in
[0,\infty)$ and $\phi$ is a KMS$_\beta$ state of $(\T C^*(\Lambda),\alpha)$.
\begin{enumerate}
\item\label{ak} Define $m^\phi=(m_v)\in [0,1]^{\Lambda^0}$ by $m_v=\phi(q_v)$. Then $m^\phi$ is a probability measure on $\Lambda^0$, and for every subset $K$ of $\{1,\ldots,k\}$ we have $\prod_{i\in K}(1-e^{-\beta r_i}A_i)m^\phi\geq 0$ pointwise.
\item\label{whenCKk} The KMS${}_\beta$ state $\phi$ factors through $C^*(\Lambda)$ if and only if $A_i m^\phi = e^{\beta r_i} m^\phi$ for every $1\le i \le k$.
\end{enumerate}
\end{prop}

For the proof of part \eqref{ak} we need a lemma.

\begin{lem}\label{prod=sum}
Let $K$ be a subset of $\{1,\ldots,k\}$ and $v\in \Lambda^0$. For $J\subset K$, we write
$e_J:=\sum_{j\in J}e_j$ and
\[
q_J=\sum_{\mu\in v\Lambda^{e_J}}t_\mu t_\mu^* \in \T C^*(\Lambda);
\]
we write $q_{\emptyset}:=q_v$ and $q_i=q_{\{i\}}$. Then
\begin{equation}\label{multprojs}
\prod_{i\in K}(q_v-q_i)=\sum_{J\subset K}(-1)^{|J|}q_J.
\end{equation}
\end{lem}

\begin{proof}
The Toeplitz-Cuntz-Krieger relation (T4) implies that $q_vq_i=q_i$ for all $i$. So
\begin{equation}\label{multout}
\prod_{i\in K}(q_v-q_i)=\sum_{J\subset K}(-1)^{|J|}\Big(\prod_{i\in J}q_i\Big).
\end{equation}
Suppose $\emptyset\subsetneq L\subset J\setminus\{i\}$.  Then $e_i\vee e_L=e_{L\cup\{i\}}$, and (T5) gives
\begin{align*}
q_iq_L&=\sum_{\mu\in v\Lambda^{e_i}}\sum_{\nu\in v\Lambda^{e_L}}t_\mu(t_\mu^*t_\nu)t_\nu^*
=\sum_{\mu\in v\Lambda^{e_i}}\sum_{\nu\in v\Lambda^{e_L}}t_\mu \Big( \sum_{(\alpha,\beta) \in\Lambda^{\min}(\mu,\nu)}t_\alpha t_\beta^* \Big)  t_\nu^*\\
&=\sum_{\mu\in v\Lambda^{e_i}}\sum_{\nu\in v\Lambda^{e_L}}\sum_{(\alpha,\beta) \in\Lambda^{\min}(\mu,\nu)} t_{\mu\alpha} t_{\nu\beta}^*
=\sum_{\gamma\in v\Lambda^{e_{L\cup\{i\}}}}t_\gamma t_\gamma^*=q_{L\cup\{i\}}.
\end{align*}
We deduce that $\prod_{i\in J}q_i=q_J$, and now \eqref{multout} implies
\eqref{multprojs}.
\end{proof}

\begin{proof}[Proof of Proposition~\ref{propsofm}] \eqref{ak} We have $m_v\geq 0$ because $\phi$ is positive. The measure $m^\phi$ is a probability measure on
$\Lambda^0$ because $\sum_{v\in \Lambda^0} q_v=1$ in $\T C^*(\Lambda)$, and hence
\[
1=\phi(1)=\sum_{v\in \Lambda^0}\phi(q_v)=\sum_{v\in \Lambda^0}m_v.
\]
Now let $K$ be a subset of $\{1,\ldots,k\}$ and choose $v\in \Lambda^0$. Since (T4) gives $q_v\geq
q_i$ and all the range projections commute, we have $\prod_{i\in K}(q_v-q_i)\geq 0$ in $\T
C^*(\Lambda)$, and $\phi\big(\prod_{i\in K}(q_v-q_i)\big)\geq 0$. We calculate using Lemma~\ref{prod=sum} and then \eqref{propKMSbetak}:
\begin{align*}
0&\leq \phi\Big(\sum_{J\subset K}(-1)^{|J|}q_J\Big)=\sum_{J\subset K}(-1)^{|J|}\phi(q_J)
 =\sum_{J\subset K}(-1)^{|J|}\Big(\sum_{\mu\in v\Lambda^{e_J}}\phi(t_\mu t_\mu^*)\Big)\\
&=\sum_{J\subset K}(-1)^{|J|}\Big(\sum_{\mu\in v\Lambda^{e_J}}e^{-\beta r\cdot e_J}m_{s(\mu)}\Big)
 =\sum_{J\subset K}(-1)^{|J|}e^{-\beta r\cdot e_J}\Big(\sum_{w\in\Lambda^{0}}|v\Lambda^{e_J}w|m_{w}\Big)\\
&=\sum_{J\subset K}(-1)^{|J|}e^{-\beta r\cdot e_J}\Big(\sum_{w\in\Lambda^{0}} A^{e_J}(v,w)m_{w}\Big)
 =\sum_{J\subset K}(-1)^{|J|}e^{-\beta r\cdot e_J}\big(A^{e_J}m^\phi\big)_v\\
&=\sum_{J\subset K}(-1)^{|J|}\Big(\Big(\prod_{j\in J}e^{-\beta r_j}A_j\Big)m^\phi\Big)_v
 =\Big(\Big(\prod_{i\in K}\big(1-e^{-\beta r_i}A_i\big)\Big)m^\phi\Big)_v,
\end{align*}
as required.

\eqref{whenCKk} As pointed out in \cite[Remark~1.6(iii)]{KP}, it suffices to check the Cuntz-Krieger relation (CK) on the generators $n=e_i$ of $\NN^k$. Thus $C^*(\Lambda)$ is the quotient of $\T C^*(\Lambda)$ by the ideal $J$ generated by the projections 
\[
\Big\{q_v - \sum_{\lambda \in
v\Lambda^{e_i}} t_\lambda t^*_\lambda:\text{$v \in \Lambda^0$ and $1 \le i \le k$}\Big\}.
\]
For each generating projection of $J$, a calculation using \eqref{charKMSbetak} gives
\begin{align}\label{factoriff}
\phi\Big(q_v-\sum_{\lambda \in v\Lambda^{e_i}} t_\lambda t^*_\lambda\Big)
&=m_v-\sum_{\lambda \in v\Lambda^{e_i}}e^{-\beta r_i}\phi(q_{s(\lambda)})=m_v-\sum_{w\in \Lambda^0}e^{-\beta r_i}|v\Lambda^{e_i}w|\phi(q_w)\\
&=m_v-\sum_{w\in \Lambda^0}e^{-\beta r_i}A_i(v,w)m_w=m_v-e^{-\beta r_i}(A_im^\phi)_v.\notag
\end{align}
If $\phi$ factors through a state of $C^*(\Lambda)$, then the left-hand side of \eqref{factoriff} vanishes, and we have $m^\phi-e^{-\beta r_i}A_im^\phi=0$. Suppose on the other hand that $m^\phi=e^{-\beta r_i}A_im^\phi$. Then \eqref{factoriff} implies that $\phi$ vanishes on the generators of $J$. Now each of these generating projections is fixed by the action $\alpha$, and for each spanning element $a=t_\mu t^*_{\nu}$ of $\T C^*(\Lambda)$, the analytic function $f_a(z):=e^{izr\cdot (d(\mu) - d(\nu))}$ satisfies $\alpha_z(a)=f_a(z)a$. Thus Lemma~2.2 of \cite{aHLRSk=1} implies that $\phi$ vanishes on the ideal $J$, and hence factors through a state of $C^*(\Lambda)=\T C^*(\Lambda)/J$.
\end{proof}

The relations $A_i m^\phi = e^{\beta r_i} m^\phi$ in part~\eqref{whenCKk} have profound implications for the KMS states on $C^*(\Lambda)$, because they imply that $e^{\beta r_i}$ is the Perron-Frobenius eigenvalue $\rho(A_i)$. (As was first noticed by Enomoto, Fujii and Watatani in the context of $\{0,1\}$-matrices \cite{EnomotoFujiiEtAl:MJ84}.)  The \emph{subinvariance} relation $(1-e^{-\beta r_i}A_i)m^\phi\geq 0$ in part \eqref{ak}, which is equivalent to $A_i m^\phi \leq e^{\beta r_i} m^\phi$, has similar implications for the KMS states on the Toeplitz-Cuntz-Krieger algebra $\T C^*(\Lambda)$. (This was used by Exel and Laca in the context of infinite $\{0,1\}$-matrices \cite[\S12]{EL}, and later in the context of finite-graph algebras \cite{aHLRSk=1}.) 

The  Perron-Frobenius theorem is about irreducible matrices; we say that a $k$-graph $\Lambda$ is \emph{coordinatewise irreducible} if each vertex matrix $A_i$
is irreducible. Next we see that coordinatewise irreducibility restricts the possible values of the inverse temperature for KMS states.

\begin{cor}\label{restrictbeta}
Suppose that $\Lambda$ is a  coordinatewise-irreducible finite $k$-graph. Fix $r\in (0,\infty)^k$ and
define $\alpha:\RR\to \Aut \T C^*(\Lambda)$ by $\alpha_t=\gamma_{e^{itr}}$. If $\beta\in
[0,\infty)$ and there is a KMS$_\beta$ state for $(\T C^*(\Lambda),\alpha)$, then $\beta r_i\geq
\ln \rho(A_i)\geq 0$ for $1\leq i\leq k$.
\end{cor}

\begin{proof}
Define $m^\phi:\Lambda^0\to[0,\infty)$ by $m^\phi=(\phi(q_v))$. Applying Proposition~\ref{propsofm}\;\eqref{ak} to the singleton sets $K=\{i\}$ shows that $(1-e^{-\beta r_i}A_i)m^\phi\geq 0$ for $1\leq
i\leq k$. This says precisely that for each $i$, $m^\phi$ is a subinvariant vector in the sense
that $A_im^\phi\leq e^{\beta r_i}m^\phi$ pointwise, and then \cite[Theorem~1.6]{Seneta} implies
that $e^{\beta r_i}\geq \rho(A_i)$.

Since each vertex receives edges of every colour, each row in each vertex matrix $A_i$ contains at least one positive integer. Thus each row-sum of each $A_i$ is at least $1$, and Corollary~1 on \cite[page~8]{Seneta} implies that $\rho(A_i)\geq 1$ for $1\leq i\leq k$.
\end{proof}

\begin{cor}\label{restrictbeta2}
Suppose that $\Lambda$ is a  coordinatewise-irreducible finite $k$-graph. Fix $r\in (0,\infty)^k$ and
define $\alpha:\RR\to \Aut C^*(\Lambda)$ by $\alpha_t=\gamma_{e^{itr}}$. If $\beta\in
[0,\infty)$ and there is a KMS$_\beta$ state for $(C^*(\Lambda),\alpha)$, then $\beta r_i=\ln \rho(A_i)$ for $1\leq i\leq k$.
\end{cor}

\begin{proof}
Part (b) of Proposition~\ref{propsofm} implies that the vector $m^\phi:=(\phi(q_v))$ satisfies $A_im^\phi=e^{\beta r_i}m^\phi$ for $1\leq i\leq k$. Since each $A_i$ is irreducible, the Perron-Frobenius theorem implies that $e^{\beta r_i}=\rho(A_i)$ for all $i$, and taking logs gives the result.
\end{proof}

\begin{rmk}\label{pref} 
Corollaries~\ref{restrictbeta} and~\ref{restrictbeta2} imply that there is a relationship between the dynamics, which is determined by the vector $r\in (0,\infty)^k$, and the range of possible inverse temperatures $\beta$. For any $r$, we have $\beta r_i\geq \ln\rho(A_i)$ for large values of $\beta$, and we will identify the simplex of KMS$_\beta$ states on $(\T C^*(\Lambda),\alpha)$ in Theorem~\ref{mainthmk}. When $e^r$ is a multiple of of the vector $\rho(\Lambda):=(\rho(A_1),\cdots,\rho(A_k))$ of Perron-Frobenius eigenvalues, we call the common value $\beta_c:=r_i^{-1}\ln\rho(A_i)$ the \emph{critical inverse temperature}. Corollary~\ref{restrictbeta2} shows that if $e^r$ is a multiple of $\rho(\Lambda)$, then $\beta_c$ is the only possible inverse temperature for a KMS state on $(C^*(\Lambda),\alpha)$. In this last case, we usually normalise the action by taking $r_i=\ln\rho(A_i)$ for all $i$, and then we refer to $\alpha_t=\gamma_{\rho(\Lambda)^{it}}$ as ``the preferred dynamics''. When $\Lambda$ has a single vertex, this preferred dynamics was studied by Yang \cite{Yang:xx09, Yang}. 

\end{rmk}

\section{A characterisation of KMS states for large inverse temperatures}\label{sec:largebeta}

For large $\beta$ we have $\beta r_i>\ln\rho(A_i)$ for all $i$, and we can significantly strengthen the characterisation of KMS$_\beta$ states on the Toeplitz algebra in Proposition~\ref{idKMSbetak}\,\eqref{preak} by shedding the assumption of rational independence. We use this extra strength in our computation of KMS states in \S\ref{sec:toeplitzKMS}, and it is also of some intrinsic interest (see Remark~\ref{extrasym} below).

\begin{thm}\label{improvedidKMS}
Suppose that $\Lambda$ is a finite $k$-graph with no sources. Let $r \in (0,\infty)^k$, let $\gamma : \TT^k \to \Aut(\T C^*(\Lambda))$ be the gauge action, and
define $\alpha : \RR \to \Aut(\T C^*(\Lambda))$ by $\alpha_t = \gamma_{e^{itr}}$. Suppose that $\beta\in
(0,\infty)$ satisfies $\beta r_i>\ln\rho(A_i)$ for $1\leq i\leq k$. Then a state $\phi$ on $\T C^*(\Lambda)$ is a KMS$_\beta$ state for $\alpha$ if and only if
    \begin{equation}\label{charKMSbetak2}
    \phi(t_\mu t_\nu^*)=\delta_{\mu,\nu}e^{-\beta r\cdot d(\mu)}\phi(q_{s(\mu)})\quad\text{for all $\mu,\nu\in \Lambda$.}
    \end{equation}
\end{thm}

The assumption of rational independence on the $r_i$ was used in the last paragraph of the proof of Proposition~\ref{idKMSbetak}\,\eqref{preak} to see that $d(\mu)\not= d(\nu)$ implies $r\cdot d(\mu)\not= r\cdot d(\nu)$, and the argument carries over without that assumption when $r\cdot d(\mu)\not= r\cdot d(\nu)$. However, it is now possible that $d(\mu)\not= d(\nu)$ and $r\cdot d(\mu)=r\cdot d(\nu)$,  and then \eqref{twotimeKMS} does not imply that $\phi(t_\mu t_\nu^*)=0$. The next two lemmas will help us deal with this situation.

\begin{lem}\label{CSestimate}
Suppose that $\phi$ is a KMS$_\beta$ state of $(\T C^*(\Lambda),\alpha)$ for some $\beta>0$, and that $\mu,\nu\in\Lambda$ satisfy $s(\mu)=s(\nu)$ and $r\cdot d(\mu)=r\cdot d(\nu)$. Then $\phi(t_\mu t_\mu^*)=\phi(t_\nu t_\nu^*)$ and $|\phi(t_\mu t_\nu^*)|\leq \phi(t_\mu t_\mu^*)$.
\end{lem}

\begin{proof}
We calculate, using the KMS condition and then the hypothesis $r\cdot d(\mu)=r\cdot d(\nu)$:
\begin{align*}
\phi(t_\mu t_\mu^*)&=\phi(t_\mu q_{s(\mu)} t_\mu^*)=\phi(t_\mu t_\nu^* t_\nu t_\mu^*)\\
&=e^{-\beta r\cdot(d(\mu)-d(\nu))}\phi(t_\nu t_\mu^* t_\mu t_\nu^*)=\phi(t_{\nu}t_{\nu}^*).
\end{align*}
Now the Cauchy-Schwarz inequality gives
\[
|\phi(t_\mu t_\nu^*)|^2\leq \phi(t_\mu t_\mu^*)\phi(t_\nu t_\nu^*)=\phi(t_\mu t_\mu^*)^2.\qedhere
\]
\end{proof}

\begin{lem}\label{lemmagic}
Suppose that $\phi$ is a KMS$_\beta$ state of $(\T C^*(\Lambda),\alpha)$ for some $\beta>0$, and that $\mu,\nu\in \Lambda$ satisfy $s(\mu)=s(\nu)$. 
\begin{enumerate}
\item\label{mostly0}
If $\lambda\in \Lambda$ satisfies $\Lambda^{\min}(\mu\lambda,\nu\lambda)=\emptyset$ then $\phi(t_{\mu\lambda}t_{\nu\lambda}^*)=0$.
\item\label{Aidansmagicform} Let  $n:=(d(\mu)\vee d(\nu))-d(\mu)$. For $j\in \NN$ we have
\begin{equation}\label{eqAidansform}
\phi(t_\mu t_\nu^*)=\sum_{\lambda\in s(\mu)\Lambda^{jn}}\phi(t_{\mu\lambda}t_{\nu\lambda}^*).
\end{equation}
\end{enumerate}
\end{lem}

\begin{proof}
\eqref{mostly0} If $\Lambda^{\min}(\mu\lambda,\nu\lambda)=\emptyset$, then the Toeplitz-Cuntz-Krieger relation (T5) implies that $t_{\nu\lambda}^*t_{\mu\lambda}=0$, and the KMS condition gives
\[
\phi(t_{\mu\lambda}t_{\nu\lambda}^*)=e^{-\beta r\cdot d(\mu\lambda)}\phi(t_{\nu\lambda}^*t_{\mu\lambda})=0. 
\]

We will prove \eqref{Aidansmagicform} by induction on $j$. It is trivially true for $j=0$. Suppose we have \eqref{eqAidansform} for some $j\geq 0$. Part~\eqref{mostly0} says that in the sum on the right of \eqref{eqAidansform}, we need only consider $\lambda\in s(\mu)\Lambda^{jn}$ such that $\Lambda^{\min}(\mu\lambda,\nu\lambda)$ is nonempty. The KMS condition implies that 
\[
\phi(t_{\nu\lambda}t_{\nu\lambda}^*t_{\mu\lambda}t_{\nu\lambda}^*)=
e^{-\beta r\cdot (d(\nu\lambda)-d(\nu\lambda))}\phi(t_{\mu\lambda}t_{\nu\lambda}^*t_{\nu\lambda}t_{\nu\lambda}^*)=\phi(t_{\mu\lambda}t_{\nu\lambda}^*).
\]
Thus (T5) gives
\begin{align*}
\phi(t_{\mu\lambda}t_{\nu\lambda}^*)
&=\phi(t_{\nu\lambda}(t_{\nu\lambda}^*t_{\mu\lambda})t_{\nu\lambda}^*)\label{TCKindstep}\\
&=\sum_{(\eta,\zeta)\in\Lambda^{\min}(\nu\lambda,\mu\lambda)}\phi(t_{\nu\lambda\eta}t_{\nu\lambda\zeta}^*)\notag\\
&=\sum_{(\eta,\zeta)\in\Lambda^{\min}(\nu\lambda,\mu\lambda)}\phi(t_{\mu\lambda\zeta}t_{\nu\lambda\zeta}^*)\notag.
\end{align*}
Combining this with  the induction hypothesis gives
\begin{equation}\label{finishoff}
\phi(t_{\mu\lambda}t_{\nu\lambda}^*)=\sum_{\lambda\in s(\mu)\Lambda^{jn}}\phi(t_{\mu\lambda}t_{\nu\lambda}^*)
=\sum_{\lambda\in s(\mu)\Lambda^{jn}}\sum_{(\eta,\zeta)\in\Lambda^{\min}(\nu\lambda,\mu\lambda)}\phi(t_{\mu\lambda\zeta}t_{\nu\lambda\zeta}^*). 
\end{equation}
For $(\eta,\zeta)\in\Lambda^{\min}(\nu\lambda,\mu\lambda)$, we have
\[
d(\mu\lambda)+d(\zeta)=d(\nu\lambda)\vee d(\mu\lambda)=(d(\mu)\vee d(\nu))+d(\lambda),
\]
which implies $d(\zeta)=(d(\mu)\vee d(\nu))-d(\mu)=n$. Thus $d(\lambda\zeta)=(j+1)n$. 

Now suppose $\tau\in s(\mu)\Lambda^{(j+1)n}$ and $\phi(t_{\mu\tau}t_{\nu\tau}^*)\not=0$. Then part~\eqref{mostly0} implies that there exists $(\gamma,\delta)\in\Lambda^{\min}(\mu\tau, \nu\tau)$, and  then $\mu\tau\gamma=\nu\tau\delta$.  But then with $\lambda:=\tau(0,jn)$, the paths $\zeta:=\tau(jn,(j+1)n)$ and $\eta:=(\tau\delta)(jn,jn+(d(\mu)\vee d(\nu))-d(\nu))$ give a pair $(\eta,\zeta)$ in $\Lambda^{\min}(\nu\lambda,\mu\lambda)$ such that $\tau=\lambda\zeta$. Thus \eqref{finishoff} gives
\[
\phi(t_{\mu\lambda}t_{\nu\lambda}^*)=\sum_{\tau\in s(\mu)\Lambda^{(j+1)n}}\phi(t_{\mu\tau}t_{\nu\tau}^*),
\]
and this is \eqref{eqAidansform} for $j+1$.
\end{proof}

\begin{proof}[Proof of Theorem~\ref{improvedidKMS}]
Suppose that $\phi$ is a KMS$_\beta$ state, and take $\mu,\nu\in \Lambda$ with $s(\mu)=s(\nu)$ and $d(\mu)\not= d(\nu)$. As observed before Lemma~\ref{CSestimate}, it remains for us to deal with the case where $r\cdot d(\mu)=r\cdot d(\nu)$. Since $d(\mu)\not= d(\nu)$, at least one of $(d(\mu)\vee d(\nu))-d(\mu)$ or $(d(\mu)\vee d(\nu))-d(\nu)$ is nonzero. Since $\phi(s_\nu s_\mu^*)=0\Longleftrightarrow\phi(s_\mu s_\nu^*)=0$, we may suppose that $n:=(d(\mu)\vee d(\nu))-d(\mu)$ is nonzero. Then for $j\in \NN$, Lemma~\ref{lemmagic} gives 
\[
\phi(t_\mu t_\nu^*)=\sum_{\lambda\in s(\mu)\Lambda^{jn}}\phi(t_{\mu\lambda} t_{\nu\lambda}^*).
\]
For each $\lambda\in s(\mu)\Lambda^{jn}$ we have $r\cdot d(\mu\lambda)=r\cdot d(\nu\lambda)$, and hence  Lemma~\ref{CSestimate} implies that
\begin{align*}
|\phi(t_\mu t_\nu^*)|&\leq\sum_{\lambda\in s(\mu)\Lambda^{jn}}\big|\phi(t_{\mu\lambda} t_{\nu\lambda}^*)\big|\\
&\leq\sum_{\lambda\in s(\mu)\Lambda^{jn}}\phi(t_{\mu\lambda} t_{\mu\lambda}^*)\\
&=\sum_{\lambda\in s(\mu)\Lambda^{jn}}e^{-\beta r\cdot d(\mu\lambda)}\phi(q_{s(\lambda)})\quad\text{(using~\eqref{propKMSbetak})}\\
&=\sum_{w\in\Lambda^0}|s(\mu)\Lambda^{jn}w|e^{-\beta r\cdot (d(\mu)+jn)}\phi(q_{w}).
\end{align*}
Now we recall from \S\ref{PFback} that the number $|v\Lambda^nw|$ is the $(v,w)$ entry in the matrix $A^n:=\prod_{i=1}^k A_i^{n_i}$. Thus for each $j\in\NN$, we have
\begin{equation}\label{estusingsr}
|\phi(t_\mu t_\nu^*)|\leq e^{-\beta r\cdot d(\mu)}\sum_{w\in \Lambda^0}\Big(\prod_{i=1}^{k}e^{-\beta r_ijn_i}A_i^{jn_i}\Big)(s(\mu),w)\phi(q_{w}).
\end{equation}
For each $i$ such that $n_i>0$, $e^{-\beta r_ijn_i}A_i^{jn_i}$ is the $(jn_i)$th term in the series $\sum e^{-\beta r_im}A_i^{m}$;  since $\beta r_i>\ln\rho(A_i)$ the series converges in the operator norm to $(1-e^{-\beta r_i}A_i)^{-1}$. In particular, for each such $i$, we have $e^{-\beta r_ijn_i}A_i^{n_i}\to 0$ as $j\to \infty$. Since $n\not=0$, there is at least one $i$ such that $n_i>0$. Thus as $j\to \infty$, the right-hand side of \eqref{estusingsr} converges to $0$, and we deduce that $|\phi(t_\mu t_\nu^*)|=0$.
\end{proof}

\begin{rmk}\label{extrasym}
The property \eqref{charKMSbetak} says that every KMS state factors through the core $\T C^*(\Lambda)^\gamma$. By \cite[Proposition~8.12.4]{P}, every KMS state is invariant for the action $\alpha$. Thus if the subgroup $e^{i\RR r}$ is dense in $\TT^k$, which is the case when the $r_i$ are rationally independent, continuity of the gauge action $\gamma$ implies that every KMS state of $(\T C^*(\Lambda),\alpha)$ is $\gamma$-invariant. Not all vectors $r$ give dense subgroups of $\TT^k$, but \eqref{charKMSbetak} still says that KMS states are $\gamma$-invariant. In other words, the set of KMS states has more symmetry than mere $\alpha$-invariance implies. The existence of extra symmetry has been an important factor in other computations of KMS states, most notably in the analysis of Bost and Connes, where an action of a large Galois group on the set of KMS states was used in the proof of uniqueness at the critical inverse temperature (see \cite[Lemma~27]{BC}, or \cite[Lemma~A.3]{LR} for a simpler proof).
\end{rmk}

\section{KMS states on the Toeplitz algebra at large inverse temperatures}\label{sec:toeplitzKMS}

\begin{thm}\label{mainthmk}
Let $\Lambda$ be a finite $k$-graph without sources, and let $A_i$ be the vertex matrices of
$\Lambda$. Suppose that $r\in (0,\infty)^k$ satisfies $\beta r_i>\ln\rho(A_i)$ for $1\leq i\leq k$, and define $\alpha:\RR\to \Aut \T
C^*(\Lambda)$ by $\alpha_t:=\gamma_{e^{itr}}$. 
\begin{enumerate}
\item\label{bk} For $v\in \Lambda^0$, the series $\sum_{\mu\in \Lambda v}e^{-\beta r\cdot d(\mu)}$ converges with sum $y_v\geq 1$. Set $y = (y_v) \in [1,\infty)^{\Lambda^0}$, and consider $\epsilon\in [0,\infty)^{\Lambda^0}$. Then $m:=\prod_{i=1}^k(1-e^{-\beta r_i}A_i)^{-1}\epsilon$ satisfies $A_im\leq e^{\beta r_i}m$ for all $i$; $m$ is a
    probability measure on $\Lambda^0$ if and only if $\epsilon\cdot y=1$.
\item\label{ck} Suppose $\epsilon\in [0,\infty)^{\Lambda^0}$ satisfies $\epsilon\cdot y=1$, and
    set $m:=\prod_{i=1}^k(1-e^{-\beta r_i}A_i)^{-1}\epsilon$. Then there is a KMS$_\beta$ state
    $\phi_\epsilon$ of $(\T C^*(\Lambda), \alpha)$ satisfying
    \begin{equation}\label{charKMSepk}
        \phi_\epsilon(t_\mu t_\nu^*)=\delta_{\mu,\nu}e^{-\beta r\cdot d(\mu)}m_{s(\mu)}.
    \end{equation}
\item\label{dohk} The map $\epsilon\mapsto\phi_\epsilon$ is an affine isomorphism of
    \[
    \Sigma_\beta:=\{\epsilon\in [0,\infty)^{\Lambda^0}:\epsilon\cdot y=1\}
    \]
    onto the simplex of KMS${}_\beta$ states of $(\T C^*(\Lambda), \alpha)$. The inverse of this
    isomorphism takes a KMS$_\beta$ state $\phi$ to $\prod_{i=1}^k(1-e^{-\beta r_i}A_i)m^\phi$,
    where $m^\phi=(\phi(q_v))$.
\end{enumerate}
\end{thm}

\begin{proof}[Proof of Theorem~\ref{mainthmk}\,\eqref{bk}]
Let $v\in \Lambda^0$. Since $A^n(w,v):=|w\Lambda^n v|$, we have
\begin{align}\label{calcyk}
\sum_{\mu\in \Lambda v}e^{-\beta r\cdot d(\mu)}&=\sum_{n\in \NN^k}\sum_{\mu\in \Lambda^nv}e^{-\beta r\cdot n}\\
&=\sum_{n\in \NN^k}\sum_{w\in \Lambda^0}e^{-\beta r\cdot n}|w\Lambda^nv|\notag\\
&=\sum_{n\in \NN^k} \sum_{w\in \Lambda^0} e^{-\beta r\cdot  n}A^n(w,v).\notag
\end{align}
Lemma~\ref{resolventk} implies that $\sum_{n=0}^\infty e^{-\beta r\cdot n}A^n$ converges in the
operator norm. So for every fixed $w\in \Lambda^0$ the series $\sum_{n=0}^\infty e^{-\beta r\cdot
n}A^n(w,v)$ converges. Thus the last sum in \eqref{calcyk} converges, and the sum is at least $1$
because all the $e^{-\beta n}A^n(v,w)$ are non-negative and $e^{-\beta 0}A^0(v,v)=1$.

Because $A$ is a non-negative matrix, the infinite expansion $m=\sum_{n\in\NN^k} e^{-\beta r\cdot
n}A^n\epsilon$ from Lemma~\ref{resolventk} shows that $m\geq 0$, and we use the same expansion to
compute
\begin{align*}
m(\Lambda^0)&=\sum_{v\in \Lambda^0}m_v=\sum_{v\in \Lambda^0}\sum_{n\in\NN^k} \big(e^{-\beta r\cdot n}A^n\epsilon\big)_v\\
&=\sum_{v\in \Lambda^0}\Big(\Big(\sum_{n\in\NN^k} e^{-\beta r\cdot n}A^n\Big)\epsilon\Big)_v
=\sum_{v\in \Lambda^0}\sum_{n\in\NN^k} \sum_{w\in \Lambda^0}e^{-\beta r\cdot n}A^n(v,w)\epsilon_w\\
&=\sum_{w\in \Lambda^0}\epsilon_w\Big(\sum_{v\in \Lambda^0}\sum_{n\in \NN^k} e^{-\beta r\cdot n}|v\Lambda^n w|\Big)
=\sum_{w\in \Lambda^0}\epsilon_w\Big( \sum_{\mu\in \Lambda w} e^{-\beta r\cdot d(\mu)} \Big)\notag\\
&=\epsilon\cdot y,
\end{align*}
which gives the last statement in \eqref{bk}.
\end{proof}

\begin{proof}[Proof of Theorem~\ref{mainthmk}\,\eqref{ck}]
To construct $\phi_\epsilon$, we use the path representation $\pi_{T,Q}$ of $\T C^*(\Lambda)$ on
$\ell^2(\Lambda)$. We define weights
\begin{equation}\label{defDeltak}
\Delta_\lambda=e^{-\beta r\cdot d(\lambda)}\epsilon_{s(\lambda)},
\end{equation}
and aim to define $\phi_\epsilon$ by
\begin{equation}\label{claimstatek}
\phi_\epsilon(a)=\sum_{\mu\in \Lambda}\Delta_\mu(\pi_{T,Q}(a)h_\mu\,|\, h_\mu)\quad\text{for $a\in \T C^*(\Lambda)$.}
\end{equation}

To see that \eqref{claimstatek} defines a state, we need to show that $\sum_{\mu\in
\Lambda}\Delta_\mu=1$. We begin by fixing $v\in \Lambda^0$ and computing
\begin{align}\label{calcsumk}
\sum_{\mu \in v \Lambda} \Delta_\mu
    &= \sum_{n\in \NN^k} \sum_{\mu \in v \Lambda^n} e^{-\beta r\cdot n}\epsilon_{s(\mu)}
     = \sum_{n\in \NN^k} e^{-\beta r\cdot n}\Big(\sum_{w\in \Lambda^0}\sum_{\mu\in v\Lambda^nw}\epsilon_{w}\Big)\\
    &= \sum_{n\in \NN^k} e^{-\beta r\cdot n}\Big(\sum_{w\in \Lambda^0}A^n(v,w)\epsilon_{w}\Big)
     = \sum_{n\in \NN^k} (e^{-\beta r\cdot n}A^n\epsilon)_{v};\notag
\end{align}
Lemma~\ref{resolventk} implies that this last series converges with sum $m_v=\big(\prod_{i=1}^k(1-e^{-\beta
r_i}A_i)^{-1}\epsilon\big)_v$. Part~\eqref{bk} implies that $m$ is a probability measure, giving
$\sum_{\mu\in \Lambda}\Delta_\mu=\sum_{v\in \Lambda^0}m_v=1$. Thus the formula \eqref{claimstatek}
defines a state $\phi_\epsilon$ on $\T C^*(\Lambda)$.

To prove that $\phi_\epsilon$ is a KMS$_\beta$ state, we verify \eqref{charKMSbetak}. For $\lambda
\in \Lambda$ we have
\[
(\pi_{T,Q}(t_\mu t^*_\nu) h_\lambda \, |\,  h_\lambda) = (T_\nu^* h_\lambda \, |\,  T^*_\mu h_\lambda)= \begin{cases}
1 &\text{ if $\lambda = \mu\lambda' = \nu\lambda'$} \\
0 &\text{ otherwise.}
\end{cases}
\]
By unique factorisation, $\mu\lambda' = \nu\lambda'$ implies $\mu = \nu$, and hence
$\phi_\epsilon(t_\mu t^*_\nu) = 0$ when $\mu \not= \nu$. If $\mu = \nu$, then the calculation
\eqref{calcsumk} gives
\begin{align*}
\phi_\epsilon(t_\mu t^*_\mu)
&= \sum_{\lambda \in \Lambda} \Delta_\lambda (T^*_\mu h_\lambda\, |\, T_\mu^*  h_\lambda)
= \sum_{\lambda=\mu\lambda'} e^{-\beta r\cdot d(\mu\lambda')}\epsilon_{s(\lambda')}\\
&=e^{-\beta r\cdot d(\mu)}\sum_{\lambda'\in s(\mu)\Lambda}\Delta_{\lambda'}=e^{-\beta r\cdot d(\mu)}m_{s(\mu)}.
\end{align*}
The calculation \eqref{calcsumk} gives
\[
\phi_\epsilon(q_{s(\mu)})=\sum_{\lambda\in \Lambda}\Delta_\lambda (Q_{s(\mu)}h_\lambda\,|\, h_\lambda)= \sum_{\lambda\in s(\mu)\Lambda}\Delta_\lambda= m_{s(\mu)}
\]
and we deduce that $\phi_\epsilon$ satisfies \eqref{charKMSbetak}. Now Proposition~\ref{idKMSbetak}\,\eqref{preak} implies that $\phi_\epsilon$ is a
KMS$_\beta$ state.
\end{proof}

\begin{proof}[Proof of Theorem~\ref{mainthmk}\,\eqref{dohk}]
Suppose $\phi$ is a KMS$_\beta$ state. Then Proposition~\ref{propsofm}\,\eqref{ak} implies that
$m^\phi=(\phi(q_v))$ is a probability measure satisfying $\epsilon:=\prod_{i=1}^k(1-e^{-\beta
r_i}A_i)m^\phi\geq 0$. Then $m:=\prod_{i=1}^k(1-e^{-\beta r_i}A_i)^{-1}\epsilon$ is just $m^\phi$,
and comparing the formula \eqref{charKMSbetak2} for $\phi$ in Theorem~\ref{improvedidKMS} with \eqref{charKMSepk} shows that $\phi=\phi_\epsilon$. So $F:\epsilon\mapsto \phi_\epsilon$ is surjective.

The formula \eqref{charKMSepk} also shows that $F$ is
injective, and that $F$ is weak* continuous from $\Sigma_\beta\subset \RR^{\Lambda^0}$ to the state
space of $\T C^*(\Lambda)$. Thus $F$ is a homeomorphism of the compact space $\Sigma_\beta$ onto
the simplex of KMS$_\beta$ states. The formulas \eqref{defDeltak} and \eqref{claimstatek} show that
$F$ is affine, and the formula for the inverse follows from the proof of surjectivity.
\end{proof}

\section{KMS states at the critical inverse temperature}\label{sec:CK-KMS}

 The next proposition follows from Lemma~\ref{lem:common F-B
evctr}.

\begin{prop}\label{cor:k-graph PR evctr}
Let $\Lambda$ be a coordinatewise-irreducible finite $k$-graph. Then the unimodular
Perron-Frobenius eigenvectors of the $A_i$ are equal for $1\leq i\leq k$; we denote this vector
$x^\Lambda$ and call it the unimodular Perron-Frobenius eigenvector of $\Lambda$. Write
$\rho(\Lambda)$ for the vector $(\rho(A_i))\in [0,\infty)^k$. Then for $n\in \NN^k$, $x^\Lambda$ is
the Perron-Frobenius eigenvector of $A^n:=\prod_{i=1}^kA_i^{n_i}$ with eigenvalue $\rho(\Lambda)^n$
(defined using multi-index notation).
\end{prop}

As we saw in Remark~\ref{pref}, Corollary~\ref{restrictbeta} suggests that there is a preferred dynamics on $\T
C^*(\Lambda)$ in which 
\[
r= \ln\rho(\Lambda):=(\ln\rho(A_1),\cdots,\ln\rho(A_k)) \in (0,\infty)^k, 
\]
and now we have $\beta r_i>\ln\rho(A_i)$ for all $i$ if and only if $\beta>1$. In other words, the critical inverse temperature $\beta_c$ is $1$. 

Provided the $\ln\rho(A_i)$ are rationally independent, we can describe
all the KMS states for the preferred dynamics.   

\begin{thm}\label{criticalbeta}
Suppose that $\Lambda$ is a coordinatewise-irreducible finite $k$-graph such that the $\ln\rho(\Lambda_j)$ are rationally independent. Let $\gamma:\TT^k\to \Aut C^*(\Lambda)$ be the
gauge action, and define $\alpha:\RR\to \Aut \T C^*(\Lambda)$ by
\[
\alpha_t=\gamma_{e^{it\ln \rho(\Lambda)}}=\gamma_{\rho(\Lambda)^{it}}:={\textstyle\prod_{j=1}^k\gamma_{\ln\rho(\Lambda_j)^{it}}}.
\]
Let $x^\Lambda$ be the unimodular Perron-Frobenius eigenvector for $\Lambda$. Then there is a unique KMS$_1$ state
$\phi$ for $(\T C^*(\Lambda),\alpha)$. The state $\phi$ satisfies
\begin{equation}\label{KMS1}
\phi(t_\mu t^*_\nu) = \delta_{\mu,\nu} \rho(\Lambda)^{-d(\mu)} x_{s(\mu)}\quad\text{ for all }\mu,\nu \in \Lambda,
\end{equation}
and this state $\phi$ factors through a state $\bar\phi$ of the quotient $C^*(\Lambda)$. The state $\bar\phi$ is
the only KMS state for $(C^*(\Lambda),\alpha)$.
\end{thm}

\begin{proof}
Choose a decreasing sequence $\{\beta_n\}$ such that $\beta_n\to 1$. For each $n$, the vector $x^\Lambda$
is a probability measure satisfying $A_ix^\Lambda = \rho(A_i) x^\Lambda \leq e^{\beta_n} x^\Lambda$ for all $i$. Since
$\beta_n\ln\rho(A_i) > \ln\rho(A_i)$, applying Theorem~\ref{mainthmk}\,\eqref{ck} to
$\epsilon := \prod_{i=1}^k(1-e^{-\beta_n}A_i)x^\Lambda$ gives a KMS$_{\beta_n}$ state $\phi_n$ satisfying
\begin{equation}\label{phin}
    \phi_n(t_\mu t^*_\nu) = \delta_{\mu,\nu} \rho(\Lambda)^{-\beta_n d(\mu)}x^\Lambda_{s(\mu)}.
\end{equation}
Since the state space of $\T C^*(\Lambda)$ is weak* compact we may assume by passing to a
subsequence that the sequence $\{\phi_{n}\}$ converges to a state $\phi$. Letting $n\to\infty$
in~\eqref{phin} shows that $\phi$ satisfies~\eqref{KMS1}.
Thus Proposition~\ref{idKMSbetak}\,\eqref{preak} implies that $\phi$ is a KMS${}_{1}$ state. (We could
also deduce that the limit is a KMS$_1$ state by applying \cite[Proposition~5.3.23]{BR}.)

To establish uniqueness, suppose that $\psi$ is a KMS$_{1}$ state for $(\T C^*(\Lambda),\alpha)$.
Taking $K=\{i\}$ in Proposition~\ref{propsofm}\,\eqref{ak} implies that $m^\psi=(\psi(q_v))$ is a
probability measure satisfying $(1- e^{-\ln\rho(A_i)}A_i)m^\psi\geq 0$, or equivalently
$A_im^\psi\leq \rho(A_i)m^\psi$. Now the last assertion in Theorem~1.6 of \cite{Seneta} implies
that $m^\psi$ is the Perron-Frobenius eigenvector $x^\Lambda$. Since the $\ln\rho(A_i)$ are not rationally related, Proposition~\ref{idKMSbetak}\,\eqref{preak} implies that 
\begin{equation*}
\psi(t_\mu t^*_\nu) = \delta_{\mu,\nu} \rho(\Lambda)^{-d(\mu)} x_{s(\mu)}\quad\text{ for all }\mu,\nu \in \Lambda.
\end{equation*}
This formula and \eqref{KMS1} imply that $\psi=\phi$.

Since $m^\phi = x^\Lambda$ satisfies $A_i m^\phi = \rho(A_i)m^\phi$ for each $i \le k$, Proposition~\ref{propsofm}\,\eqref{whenCKk} implies that $\phi$ factors through a state of $C^*(\Lambda)$.

To see that $\bar\phi$ is the only KMS state of $(C^*(\Lambda),\alpha)$, we suppose $\psi$ is a KMS$_\beta$ state of $(C^*(\Lambda),\alpha)$. Then $\psi\circ q$ is a KMS state of $\T C^*(\Lambda)$, and
Proposition~\ref{propsofm}\,(\ref{whenCKk}) shows that $A_im^{\psi\circ q}=\rho(A_i)^{\beta}m^{\psi\circ q}$. The Perron-Frobenius theorem implies that
$\rho(A_i)^{\beta}=\rho(A_i)$ for all $i$; since the independence of the $\ln \rho(A_i)$ implies that at least one $\rho(A_i)$ is not equal to $1$, we deduce that $\beta=1$. Thus the uniqueness of the KMS$_1$ state of $(\T C^*(\Lambda),\alpha)$ implies that $\psi=\bar\phi$.
\end{proof}

\begin{ex}\label{ex:periodic}
The hypothesis that the $\ln\rho(A_i)$ are rationally independent was not used in the proof of existence of the KMS$_1$ state, but was crucial in the proof of uniqueness. To see why, consider the $2$-graph $\Gamma$ with one vertex $v$, two blue edges $e$ and $f$ and
two red edges $a$ and $b$, and with factorisation property determined by
\begin{equation}\label{eq:factprop}
e a = a e,\quad e b = a f,\quad f a = b e\quad\text{ and }\quad f b = b f.
\end{equation}
(In the language of \cite{DavidsonPowerEtAl:JFA08}, this is
$\mathbb{F}^+_{\theta}$ for the permutation $\theta:(i,j) \mapsto (j,i)$ of $\{1,2\} \times
\{1,2\}$.)

Let $U = t_e t^*_a + t_f t^*_b \in C^*(\Gamma)$. Routine calculations show that
\begin{itemize}
\item[(a)] $t_e, t_f$ is a Cuntz family, $U$ is a unitary which commutes with $t_e$ and $t_f$, and $\{t_e, t_f, U\}$ generates $C^*(\Gamma)$; and
\item[(b)] if $S_1, S_2$ is another Cuntz family and $V$ another unitary such that $V$ commutes with the $S_i$, then $T_e := S_1$, $T_f := S_2$, $T_a := VS_1$ and $T_b := VS_2$ determines a Cuntz-Krieger $\Gamma$ family, and hence there is a homomorphism carrying $t_e$ to $S_1$, $t_f$ to $S_2$ and $U$ to $V$.
\end{itemize}
So $C^*(\Gamma)$ has the universal property of the tensor product, and there is an isomorphism $\pi$ of $C^*(\Gamma)$ onto $\Oo_2 \otimes C(\TT)$ such that $\pi(t_e)=s_1 \otimes 1$, $\pi(t_f)=s_2 \otimes 1$ and $\pi(U)=1 \otimes z$.

Corollary~B.12 of \cite{tfb} implies that states $\phi$ on $A$ and $\psi$ on $B$  give a state on $A \otimes B$ such that $(\phi \otimes \psi)(a \otimes b) = \phi(a)\psi(b)$. Let
$\phi$ be the unique KMS$_{\ln 2}$ state of $\Oo_2$. The isomorphism $\pi$ satisfies $\pi\circ\alpha_t=(\gamma_{2^{it}} \otimes
\operatorname{id})\circ\pi$, where $\gamma$ is the gauge action on $\Oo_2$. So for any state $\psi$ of
$C(\TT)$, the state $(\phi \otimes \psi) \circ \pi$ is a KMS$_1$ state for $C^*(\Gamma, \alpha)$.
In particular, $(C^*(\Gamma),\alpha)$ has many KMS$_1$ states.
\end{ex}

\section{Ground states and \texorpdfstring{KMS$_\infty$}{KMS-infinity}-states}\label{sec:ground}

\begin{prop}\label{prp:ground}
Let $\Lambda$ be a finite $k$-graph. Suppose $r \in (0, \infty)^k$ and define
$\alpha:\RR\to\Aut\T C^*(\Lambda)$ by $\alpha_t=\gamma_{e^{itr}}$. For each probability measure
$\epsilon$ on $\Lambda^0$ there is a unique KMS$_\infty$ state $\phi_\epsilon$ of $(\T C^*(\Lambda), \alpha)$
satisfying
\begin{equation}\label{defphie}
\phi_\epsilon(t_\mu t_\nu^*)=
\begin{cases}
\epsilon_v&\text{if $\mu=\nu=v \in \Lambda^0$} \\
0&\text{otherwise.}\\
\end{cases}
\end{equation}
The map $\epsilon\mapsto \phi_\epsilon$ is an affine isomorphism of the simplex of probability
measures on $\Lambda^0$ onto the ground states of $(\T C^*(\Lambda),\alpha)$. In particular, every
ground state of $(\T C^*(\Lambda),\alpha)$ is a KMS$_\infty$ state.
\end{prop}
\begin{proof}
Suppose $\epsilon$ is a probability measure on $\Lambda^0$. Choose $\beta_j\in (0,\infty)$ such that $\beta_j\to\infty$ and $e^{\beta_j r_i} > \rho(A_i)$ for every $i$ and $j$. For each $j$, let $y^j \in [1,\infty)^{\Lambda^0}$ be the vector of
Theorem~\ref{mainthmk}\,\eqref{bk} for $\beta=\beta_j$, so that $y^j$ has entries $y^j_v = \sum_{\mu\in \Lambda
v}e^{-\beta_j r\cdot d(\mu)}$. Define $\epsilon^j \in (0,\infty)^{\Lambda^0}$ by
$\epsilon^j_v:=\epsilon_v(y_v^j)^{-1}$, and let $\phi_j$ be the KMS$_{\beta_j}$ state
$\phi_{\epsilon^j}$ of $(\T C^*(\Lambda),\alpha)$ described in Theorem~\ref{mainthmk}\,\eqref{ck}.
Since the state space is compact, we may assume by passing to a subsequence that the $\phi_j$
converge to a state $\phi_\epsilon$.

Set $m^j := \prod^k_{i=1}(I - e^{-\beta_jr_i}A_i)^{-1} \epsilon^j$, and fix $\mu, \nu \in \Lambda$.
Then
\[
\phi_j(t_\mu t^*_\nu) = \delta_{\mu,\nu} e^{-\beta_j r \cdot d(\mu)} m^j_{s(\mu)}.
\]
This is always $0$ if $\mu\not=\nu$, so suppose that $\mu = \nu$. If $d(\mu) \not= 0$, then
$e^{-\beta_j r \cdot d(\mu)} \to 0$ because each $r_i > 0$, and hence $\phi_\epsilon(t_\mu t^*_\nu)
= \lim_j \phi_j(t_\mu t^*_\mu) = 0$.  So suppose that $\mu = \nu = v$ is a vertex. An application of the dominated convergence theorem
shows that $y^j_v\to 1$ as $j\to\infty$. Thus $\epsilon^j_v \to \epsilon_v$. Since $\prod^k_{i=1}(I
- e^{-\beta_jr_i}A_i)^{-1} \to I$ in the operator norm as $j \to \infty$,
\[
\phi_\epsilon(q_v) = \lim_j \phi_j(q_v) = \lim_j \Big(\prod^k_{i=1}(I - e^{-\beta_jr_i}A_i)^{-1} \epsilon\Big)_v
    = \epsilon_v.
\]
So $\phi_\epsilon$ satisfies \eqref{defphie}.

Proposition~\ref{idKMSbetak}\,(\ref{groundk}) implies that if $\phi$ is a ground state and $\epsilon_v:=\phi(q_v)$, then $\phi=\phi_\epsilon$. Thus $\epsilon\mapsto \phi_\epsilon$ maps the simplex of probability measures onto the ground states, and it is affine and injective. Since each $\phi_\epsilon$ is by construction a KMS$_\infty$ state, every ground state is a KMS$_\infty$ state.
\end{proof}


\begin{thebibliography}{aa}

\bibitem{BC} J.-B. Bost and A. Connes, Hecke algebras, type III factors  and phase transitions with spontaneous symmetry breaking in number theory, \emph{Selecta Math. (New Series)} {\bf 1} (1995), 411--457.

\bibitem{BR} O. Bratteli and D.W. Robinson, \emph{Operator Algebras and Quantum Statistical Mechanics II}, second edition, Springer-Verlag, Berlin, 1997.

\bibitem{DavidsonPowerEtAl:JFA08} K.R. Davidson, S.C. Power, and D. Yang, Atomic representations of rank 2 graph algebras, \emph{J. Funct. Anal.} \textbf{255} (2008), 819--853.

\bibitem{DY} K.R. Davidson  and D. Yang, Periodicity in rank 2 graph algebras,  \emph{Canad. J. Math.} \textbf{61} (2009), 1239--1261.

\bibitem{DS} N. Dunford and J.T. Schwartz, \emph{Linear Operators. Part I: General Theory}, Interscience (John Wiley and Sons), New York, 1958.

\bibitem{EnomotoFujiiEtAl:MJ84} M. Enomoto, M. Fujii, and Y. Watatani, KMS states for gauge action on $O_{A}$, \emph{Math. Japon.} \textbf{29} (1984), 607--619.

\bibitem{EL}  R. Exel and M. Laca, Partial dynamical systems and the KMS condition, \emph{Comm. Math. Phys.} \textbf{232} (2003),  223--277.

\bibitem{F} N.J. Fowler, Discrete product systems of Hilbert bimodules, \emph{Pacific J. Math.} \textbf{204} (2002), 335--375.

\bibitem{FR} N.J. Fowler and I. Raeburn, The Toeplitz algebra of a Hilbert bimodule, \emph{Indiana Univ. Math. J.} \textbf{48} (1999),  155--181.

\bibitem{aHLRSk=1} A. an Huef, M. Laca, I. Raeburn and A. Sims, KMS states on the $C^*$-algebras of finite graphs, preprint; arXiv:1205.2194.

\bibitem{KW} T. Kajiwara and Y. Watatani, KMS states on finite-graph $C^*$-algebras, preprint; arXiv:1007.4248.

\bibitem{KP}A. Kumjian and D. Pask, Higher rank graph $C^\ast$-algebras, \emph{New York J. Math.} \textbf{6} (2000), 1--20.

\bibitem{LN} M. Laca and S. Neshveyev, KMS states of quasi-free dynamics on Pimsner algebras, \emph{J. Funct. Anal.} {\bf 211} (2004), 457--482.

\bibitem{LR} M. Laca and I. Raeburn, Phase transition on the Toeplitz algebra of the affine semigroup over the natural numbers, \emph{Adv. Math.} \textbf{225} (2010), 643--688.

\bibitem{LRR} M. Laca, I. Raeburn and J. Ramagge, Phase transition on Exel crossed products associated to dilation matrices, \emph{J. Funct. Anal.} \textbf{261} (2011), 3633--3664.

\bibitem{P} G.K. Pedersen, \emph{$C^*$-Algebras and their Automorphism Groups}, London Math. Soc. Monographs, vol. 14, Academic Press, London, 1979.

\bibitem{RS} I. Raeburn and A. Sims, Product systems of graphs and the Toeplitz algebras of higher-rank graphs, \emph{J. Operator Theory} \textbf{53} (2005), 399--429.

\bibitem{tfb} I. Raeburn and D.P. Williams, \emph{Morita Equivalence and Continuous-Trace {$C\sp *$}-Algebras}, Mathematical Surveys and Monographs, vol.~60, American Mathematical Society, Providence, RI, 1998.

\bibitem{RobertsonSteger:JRAM99} G. Robertson  and T. Steger, Affine buildings, tiling systems and higher rank {C}untz-{K}rieger algebras, \emph{J. Reine Angew. Math.} \textbf{513} (1999), 115--144.

\bibitem{Seneta} E. Seneta, \emph{Non-Negative Matrices and Markov Chains}, second edition, Springer, 1981.

\bibitem{Yang:xx09} D. Yang, Endomorphisms and modular theory of 2-graph $C^*$-algebras, \emph{Indiana Univ. Math. J.} \textbf{59} (2010), 495--520.

\bibitem{Yang} D. Yang, Type III von Neumann algebras associated with $\mathcal{O}_\theta$, \emph{Bull. London Math. Soc.} \textbf{44} (2012), 675--686.

\end{thebibliography}
\end{document}